\documentclass[a4paper,12pt,twoside]{amsart}
\usepackage{vmargin}
\usepackage{amssymb}
\usepackage{amsfonts}
\usepackage{amsmath}
\usepackage{mathrsfs}
\usepackage[dvips]{graphicx}
\usepackage{amscd}
\usepackage{color}
\numberwithin{equation}{section}
\DeclareGraphicsExtensions{.pdf, .png, .jpg, .mps}
\usepackage{float}
\usepackage{comment}
\def\p{\partial} 
\def\no{\noindent}  
\def\io{{\infty}}

 \def\moo{C^{\io}}
\def\mooc{C^{\io}_{\textit c}}

\def\R{\mathbb R}

\def\poscal#1#2{\langle#1,#2\rangle} 
   
\def\poi#1#2{\left\{#1,#2\right\}}
\def\norm#1{\Vert#1\Vert}
\def\val#1{\vert#1\vert}

\def\l2{L^2(\R^{n})}
\def\L2{L^2(\R^{2n})}
\def\supp{\operatorname{supp}}

\def\tr#1{{^t}\!#1}

\def\hs{{\hskip15pt}}
\def\vs{\vskip.3cm}

\let\no=\noindent

\let\no=\noindent

\newtheorem{theorem}{Theorem}[section]
\newtheorem{lemma}[theorem]{Lemma}

\newtheorem{prop}[theorem]{Proposition}
\newtheorem{rem}[theorem]{Remark}
\numberwithin{equation}{section}
\begin{document}
\baselineskip=1.2\normalbaselineskip
\title
[Continuation through transversal characteristic hypersurfaces]{Unique continuation through transversal characteristic hypersurfaces} 
\author[\tiny Nicolas Lerner ]
{Nicolas Lerner}
\begin{address}{Nicolas Lerner,
Projet analyse fonctionnelle, Institut de Math\'ematiques de Jussieu,
Universit\'e Pierre-et-Marie-Curie (Paris 6),
4, place Jussieu 75252 Paris cedex, France.}
\email{lerner@math.jussieu.fr}
\end{address}
\date{\today}
\begin{abstract}
We prove a unique continuation result for an ill-posed characteristic  problem.
A model problem of this type occurs in A.D.~Ionescu \& S.~Klainerman article
(Theorem 1.1 in \cite{MR2470908}) and we extend their model-result using only geometric assumptions.
The main tools are Carleman estimates and H\"ormander's pseudo-convexity conditions.
\end{abstract}
\maketitle
\noindent
\\
\noindent 35A07--35A05--34A40
\vfill
{\color{red}\tableofcontents}
\vs
\vfill
%%%%%%%%%%%%%%%%%%%%%%%%%%%%%%%%%%%%%%%%%%%
\section{Introduction}
\subsection{Background}\label{sec11}
Cauchy uniqueness for Partial Differential Equations has already a long history
and although we do not intend to revisit the many known results on this topic, we
would like to begin this paper with recalling a few basic facts; 
 since we intend to investigate 
``initial''
hypersurfaces which could be time-like in Lorentzian geometry, we shall refrain to using  a
set of coordinates suggesting that we study an evolution equation with respect to a time
variable.
\par
Let us then consider  a linear differential operator of order $m\ge 1$
on some open subset $\Omega$ of $\R^{n}$
\begin{equation}\label{}
P(x,D_{x})=\sum_{\val \alpha\le m}a_{\alpha}(x) D_{x}^{\alpha},
\end{equation}
and the following differential inequality
\begin{equation}\label{diffineq}
\val{P(x,D_{x})u}\le C\bigl(\val{u(x)}+\dots+\val {\nabla^{m-1} u(x)}\bigr).
\end{equation}
Let us consider an oriented hypersurface $\Sigma$ defined as $\{x\in \Omega, \psi(x)=0\}$ where $\psi$ is a function in  $C^{1}(\Omega;\R)$  such that $d\psi\not=0$ at $\psi=0$.
We shall say that 
the operator $P$ has the Cauchy uniqueness across the oriented hypersurface $\Sigma$ whenever
\begin{equation}\label{cauchyun}
\eqref{diffineq}\text{ holds  in $\Omega$ and $u_{\vert \psi<0}=0$}\Longrightarrow u=0 \text{ near $\Sigma$.}
\end{equation}
Of course some more precise assumptions should be made on the regularity of $u$ and the $a_{\alpha}$ above,
at least for the expression of $a_{\alpha}D_{x}^{\alpha}u$ and \eqref{diffineq}
to make sense;
we shall  go back to these questions later on.
The hypersurface $\Sigma$ will be said non-characteristic with respect to $P$ if
\begin{equation}\label{char}
p(x,d\psi(x))\not=0\quad \text{for $x\in \Sigma$, \quad with $p(x,\xi)=\sum_{\val \alpha=m}a_{\alpha}(x) \xi^{\alpha}$.}
\end{equation}
The function $p$, which is a polynomial in   the variable $\xi$, is  called the principal symbol of the operator $P$.
\subsubsection*{Well-posed problems}
The first case of interest (and certainly the first which was  investigated) is the strictly hyperbolic case,
for which we have  $x=(t,y)\in \R\times \R^{d}$, $t$ is the time-variable, $y$ are the space-variables and the well-named initial hypersurface is given by $\{t=0\}$
and 
$$
\R\ni \tau\mapsto p(t,y;\tau, \eta)\quad \text{has $m$ distinct real roots for $\eta\in \mathbb S^{d-1}$.}
$$
In that case, the quite standard energy method
(see e.g. Chapter 23 in  \cite{MR2304165})
will provide nonetheless uniqueness
but also Hadamard well-posedness, that is conti\-nuous dependence
of the solution at time $t$ with respect to the initial data at time 0, expressed by some inequalities of  type
$$
\norm{u(t)}_{E}\le \norm{u(0)}_{F},
$$
in some appropriate functional spaces $E,F$.
Lax-Mizohata theorems
(see e.g. \cite{MR0097628}, \cite{MR0170112}),
are proving that 
well-posedness
is implying that the above roots should be real, non necessarily distinct:
Well-posedness  implies weak hyperbolicity.
\subsubsection*{Elliptic problems} In 1939, the Swedish mathematician Torsten Carleman raised the following (2D)
question in \cite{MR0000334}:
let us assume that $u$ is a $C^{2}$ function such that
\begin{equation}\label{laplace}
\val{(\Delta u)(x}\le C\bigl(\val{u(x)}+\val{\nabla u(x)}\bigr),\quad u_{\vert x_{1}<0}=0.
\end{equation}
Does that imply that $u$ vanishes near $\{x_{1}=0\}$?
Considering the roots of the polynomial of $\xi_{1}$ for $\xi'=(\xi_{2},\dots,\xi_{n})\in \mathbb S^{n-2}$,
$$
\xi_{1}^{2}+\val{\xi'}^{2}=0\Longleftrightarrow \xi_{1}=\pm i\val {\xi'}\in i\R^{*},
$$
we see that the Cauchy problem for the Laplace operator is ill-posed (i.e. not well-posed),
otherwise the Lax-Mizohata results would imply weak hyperbolicity
(which does not hold).
Carleman's question was 
moving into uncharted territory and his positive answer was indeed seminal by introducing 
a completely new method, strikingly different from the energy method and based upon some weighted inequalities of type
\begin{equation}\label{carleman}
C\norm{e^{-\lambda \phi} Pv}_{L^{2}(\R^{n})}\ge \sum_{0\le j\le m-1}\lambda^{m-\frac12-j}\norm{e^{-\lambda \phi} \nabla ^{j}v}_{L^{2}(\R^{n})},
\end{equation}
for functions $v\in \mooc(\Omega)$, with a well-chosen weight $\phi$, close but different from  the function $\psi$ defining  the hypersurface $\Sigma$.
Applying  Inequality \eqref{carleman} to a regularization $v$ of $\chi u$, where $\chi$ is a cutoff function yields easily that \eqref{laplace}
entails $u=0$.
\par
More generally, it is possible to prove using the same lines that  a second-order elliptic operator $P$
with real $C^{1}$
coefficients has the Cauchy uniqueness across a $C^{2}$ hypersurface in the sense of \eqref{cauchyun}
(A.~Calder\'on \cite{MR0104925}, Chapter 8 in L.~H\"ormander's 1963 book \cite{MR0161012}).
Much more refined results were obtained much later for the Laplace operator
 by D.~Jerison \& C.~Kenig in \cite{MR794370},
simplified and extended by C.~Sogge in \cite{MR1081811}, dealing with $L^{p}-L^{q}$ inequalities of type \eqref{carleman} with singular optimal weights, yielding stronger uniqueness properties for second-order elliptic operators with real coefficients.
\subsubsection*{The analytic dead-end}
Going back to Carleman's question displayed in the previous section,
we note that Holmgren's theorem (see e.g. Theorem 8.6.5 in \cite{MR1996773}) would give a positive answer  for an analytic equation 
replacing the inequality in \eqref{laplace}
such as
$$
\Delta u+\sum_{\val \alpha\le 1} a_{\alpha}(x) D_{x}^{\alpha}u=0,\quad a_{\alpha} \text{ analytic}.
$$
However, Holmgren's assumptions of analyticity are so strong that they are in fact quite instable:
for instance there is no way to tackle the uniqueness problem \eqref{laplace},
nor even to deal with the equation $\Delta u+Vu=0$ with a $\moo$ (and non-analytic) function $V$.
The same remark could be done about Cauchy-Kovalevskaya Theorem.
 Let us quote
L. G\aa rding in \cite{MR0117434}:
{\it It was pointed out very emphatically by Hadamard that it is not natural
to consider only analytic solutions and source functions even for an operator with
analytic coefficients. 
This reduces the interest of the Cauchy-Kovalevskaya
theorem which \dots does not distinguish between
classes of differential operators which have, in fact, very different properties such as the Lapla\-ce operator and the Wave operator.}
\subsubsection*{Calder\'on's and H\"ormander's theorems}
We consider now for simplicity a se\-cond-order differential  operator $P$
with real-valued $C^{1}$ coefficients in the principal part and bounded measurable coefficients for lower order terms.
Also we consider a $C^{2}$ hypersurface $\Sigma$ given by the equation 
$\{x\in \Omega, \psi(x)=0\}, d\psi\not= 0$ at $\Sigma$
which we shall assume to be non-characteristic with respect to $P$ (cf.\eqref{char}).
Let $x_{0}$ be given on $\Sigma$ and let $p$ be the principal symbol of $P$.
\vs
$\bullet$ In the 1959 article \cite{MR0104925}, A.~Calder\'on proved that if the characteristics are simple, then 
uniqueness holds in the sense of  \eqref{cauchyun} in a neighborhood of $x_{0}$;
Calder\'on's assumptions can be written as
\begin{equation}\label{calderon}
p(x_{0},\xi)=0,\  \xi\not=0\Longrightarrow \frac{\p p}{\p \xi}(x_{0},\xi)\cdot \frac{\p\psi}{\p x}(x_{0})\not=0.
\end{equation}
Note that with $H_{p}$ standing for the Hamiltonian vector field\footnote{
The Hamiltonian vector field $H_{p}$ of $p$ is defined by
$$
H_{p}(a)=\underbrace{\poi{p}{a}}_{\substack
{\text{Poisson}\text{ bracket}}}=\frac{\p p}{\p \xi}\cdot \frac{\p a}{\p x}
-\frac{\p a}{\p \xi}\cdot \frac{\p p}{\p x},\quad\text{so that  } H_{p}(\psi)=\frac{\p p}{\p \xi}\cdot d_{x}\psi,$$
and we have also
$H_{p}^{2}(\psi)=H_{p}(H_{p}(\psi))=\poi{p}{\poi{p}{\psi}}$.
}
 of $p$, Calder\'on's
assumption can be written as 
$$
p(x_{0},\xi)=0,\  \xi\not=0\Longrightarrow H_{p}(\psi)(x_{0},\xi)\not=0,
$$
which means that the bicharacteristic curves of $p$ are transversal to the hypersurface $\Sigma$.
\vs
$\bullet$ The above result was extended in 1963 by L.~H\"ormander who proved in \cite{MR0161012}
that uniqueness holds if we assume only
\begin{equation}\label{hormander}
p(x_{0},\xi)=H_{p}(\psi)(x_{0},\xi)=0,\  \xi\not=0\Longrightarrow H_{p}^{2}(\psi)(x_{0},\xi)<0.
\end{equation}
This author gave the name {\it pseudo-convexity }
to that property which indeed means that  bicharacteristic curves tangent to $\Sigma$
have a second-order contact with $\Sigma$ and stay ``below'' $\Sigma$, i.e. in the region where $\psi\le 0$.
\vs
Of course \eqref{calderon} is a stronger assumption than \eqref{hormander} since the latter
does not require anything at $p=0, H_{p}(\psi)\not=0$, implying thus that 
H\"ormander's result contains Calder\'on's.
On the other hand it is important to note that although \eqref{calderon}
is a two-sided result which does not take into account the orientation of $\Sigma$,  Assumption \eqref{hormander}
is dealing with an oriented hypersurface 
which requires that the characteristics tangent to $\Sigma$
should have a second-order contact with $\Sigma$ {\it and} stay ``below'' $\Sigma$  (in the region where $\psi\le 0$).
H\"ormander's uniqueness result under the pseudo-convexity assumption \eqref{hormander}
is using a geometric condition (i.e. independent of the choice of a coordinate system), does not require more than one derivative for the coefficients of the principal part
and provides uniqueness for functions $u$  satisfying a differential inequality \eqref{diffineq}: it can be used to answer to the original Carleman's question above and has some interesting stability properties which are not shared by Holmgren's theorem.
Although the Cauchy problem for $P$ could be ill-posed (it will be the case for elliptic operators with respect to any hypersurface), this result allows to obtain nevertheless uniqueness when the assumptions are satisfied,
say for an elliptic operator with simple characteristics.
The method of proof used by H\"ormander follows Carleman's idea and he  is proving an estimate of type \eqref{carleman}.
\begin{figure}[H]
%\vskip-20pt
\scalebox{0.5}{\hskip-55pt\includegraphics{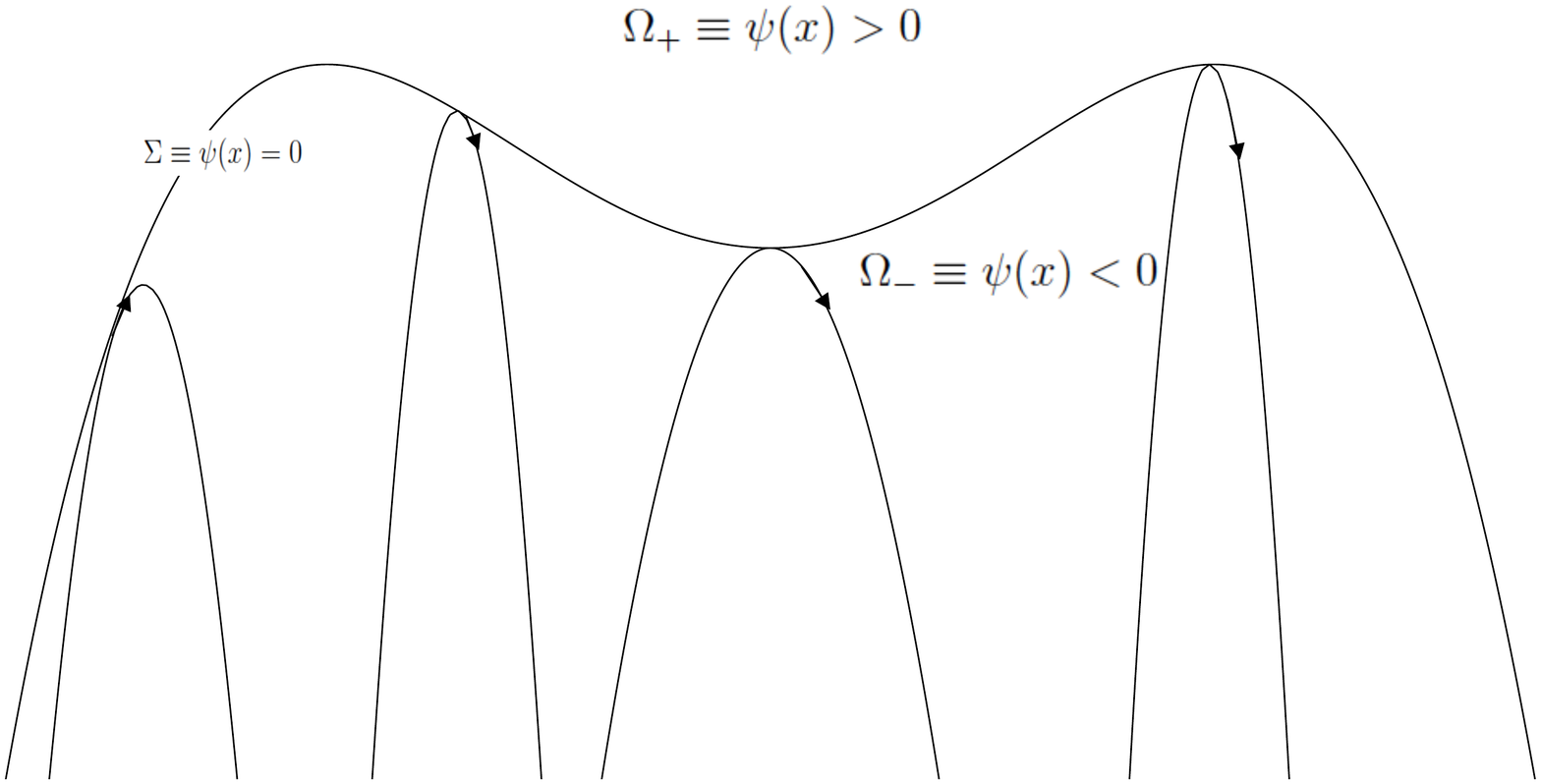}}
\caption{\small\sc Pseudo-convexity: the characteristic curves which are tangent to $\Sigma$
stay in the ``past'', below $\Sigma$.}\label{5.pic.004}
\end{figure}
\vs
\subsubsection*{Counterexamples}
Theorem 8.9.2 in \cite{MR0161012} displays a construction due to P.~Cohen: there exists a smooth non-vanishing
complex vector field in two dimensions, 
$$
\mathcal C=\p_{t}+ib(t,x) \p_{x}, \quad b\in \moo(\R^{2}),
$$
and a $\moo$ function $u$ in $\R^{2}$ with support equal to  $\{t\ge 0\}$ such that $\mathcal C u=0$.
Although the construction of that counterexample is an outstanding achievement,
it turns out that the vector field $\mathcal C$ does not satisfy  the Nirenberg-Treves condition $(P)$
(see e.g. Chapter 26 in \cite{MR1481433})  and thus is not locally solvable.
Note that this operator has simple characteristics as a first-order operator but has  complex-valued coefficients
(see also the study of complex vector fields by F.~Treves \& M. Strauss
 \cite{MR0330739} and the X.~Saint Raymond's article \cite{MR850550}).
There is a version of H\"ormander's theorem for operators with complex coefficients, but it requires a specific additional condition, the so-called principal normality
(see (28.2.8) in  \cite{MR1481433}),
which does not hold for $\mathcal C$.
\par On the other hand,
S. Alinhac and S. Baouendi constructed in \cite{MR1363855} the following  counterexample:
Let us consider the wave operator in 2-space dimension,
$$
\p_{t}^{2}-\p_{x}^{2}-\p_{y}^{2}=\square.
$$
There exists $V,u\in \moo(\R^{3})$ with
\begin{equation}\label{3.ababab}
\supp u=\{y\ge 0\},\quad \square u+Vu=0.
\end{equation}
Note that this operator is with constant coefficients, so that the characteristics are straight lines and the tangential ones are included in the boundary $y=0$, violating the pseudo-convexity hypothesis.
This problem is easily proven to be ill-posed since it is non hyperbolic with respect to the {\it time-like} hypersurface $y=0$.
The construction of this counterexample is a highly non-trivial task and this result appears as the most significant counterexample to Cauchy uniqueness.
We note in particular that this constant coefficient operator
(also of real principal type) is locally solvable, which is not the case of
$\mathcal C$.
As a result,
 the non-uniqueness property \eqref{3.ababab} is  somehow
more interesting
 for an operator having plenty of local solutions.
  The article \cite{MR683803} contains much more information on non-uniqueness results.
  \par
  A recurrent question about the counterexample \eqref{3.ababab}
was for long time if such a phenomenon could hold if $V$ does not depend on the time variable.
A negative answer was given by D. Tataru's \cite{MR1326909}, L. H\"ormander's \cite{MR2033496},  L. Robbiano \& C. Zuily  in \cite{MR1614547} who proved uniqueness for 
$\square +V(t,x,y)$ with respect to $\{y=0\}$
when $V$ is a smooth function depending analytically of the variable $t$.
Several geometric statements are given in that series of articles
which go much beyond this example.
\par
Another outstanding and still open question is linked to the Alinhac-Baouendi
counterexample: is it possible to construct a counterexample \eqref{3.ababab} when $V$ is smooth and {\it real-valued}?
The construction of  \cite{MR1363855} is using a complex-valued potential $V$ and 
one
may  conjecture that \eqref{3.ababab} with $V$ real-valued is impossible,
but a proof of this uniqueness conjecture would certainly require other tools
than Carleman's standard estimates.
\subsection{The Ionescu-Klainerman model problem}\label{secik1}
%The papers   \cite{MR2672799},
% \cite{MR2461426} are tackling several questions of Lorentzian geometry, but we focus our attention on 
%\cite{MR2470908} in the sequel.
\subsubsection*{Their statement}
 We consider the wave operator
 in $\R_{t}\times\R^{d}_{x}$ defined by
\begin{equation}\label{alember}
\square=\p_{t}^{2}-\sum_{1\le j\le d}\p_{x_{j}}^{2},
\end{equation}
and let
$ \Omega$ be the open set defined by
\begin{equation}\label{omegao}
\Omega=\{(t,x)\in \R\times \R^{d}, \quad \val{x}> 1+\val t\}.
\end{equation}
\begin{theorem}[Theorem 1.1 in \cite{MR2470908}]\label{thmik0}
 Let $u\in C^{2}(\overline{\Omega})$ vanishing on $\p \Omega$
such that the following pointwise inequality holds in $\Omega$,
 \begin{equation}\label{inequal}
\val{\square u }\le C\bigl(\val u+\val{\p_{t} u}+\val{\nabla_{x} u}\bigr),
\end{equation}
for some fixed constant $C$.
Then $u$ vanishes on $\Omega$.
 \end{theorem}
 This result contains several interesting features: in the first place, the function $u$ is defined only on $\overline \Omega$ and the equation holds in $\Omega$, whose boundary is characteristic for $\square$.
 Also $u$ is only assumed to vanish on $\p\Omega$ and no vanishing of any first derivative is required although the operator is second-order.
 \subsubsection*{Comments}
 Since we intend to provide a more general statement with an invariant hypothesis, we start with a few comments on the above result.
 First of all, we note that the boundary $\p\Omega$ is the union of two {\it transversa}l characteristic
 hypersurfaces $\Sigma_{+}, \Sigma_{-}$ with 
 $$
 \Sigma_{+}=\{(t,x), \val x=1+t, t\ge 0\},\quad  \Sigma_{-}=\{(t,x), \val x=1-t, t\le 0\},
 $$
 and 
 $
 \Omega=\{(t,x), \val x>1+t \text{ and } \val x>1-t \}.
 $
 Moreover, since the function $u$ belongs to $C^{2}(\overline \Omega)$, there exists $\widetilde u\in C^{2}(\R^{1+d})$
 such that 
 $\widetilde u_{\vert \overline \Omega}=u$.
 We can now take advantage of the fact that the boundary of $\Omega$ is the union of two characteristic hypersurfaces on which $u$ vanishes to prove that 
 $$
 v=\widetilde u\mathbf 1_{\overline\Omega}
 $$
 does satisfy the  differential inequality \eqref{inequal}
 (and is supported in $\overline \Omega$).
 It is then enough to find a pseudo-convex hypersurface with equation 
 $\{\psi=0\}$
  whose epigraph $\{\psi\ge 0\}$
  contains $\overline \Omega$
 to apply (a slight modification) of H\"ormander's  uniqueness theorem under a pseudo-convexity assumption
 to obtain the result.
 In particular, there is no need for a cutoff function supported inside $\Omega$ as in 
 \cite{MR2470908}.
 The details of our arguments are given below, but we hope that 
 these short indications could convince the reader that a geometric statement is at hand.
\subsubsection*{The characteristic Cauchy problem}
In Section \ref{sec11}, we gave results concerning the non-characteristic Cauchy problem 
for second-order operators with real coefficients.
The main reason for these restrictions is that the pseudo-convexity hypothesis
\eqref{hormander}
has a very simple geometric formulation in that framework.
However, uniqueness results and pseudo-convexity hypotheses
can be expressed even for a characteristic hypersurface (and higher order operators).
We may nevertheless say that, generically, characteristic problems do not have uniqueness
(see e.g. Theorem 5.2.1 in \cite{MR0161012})
and that the pseudo-convexity assumption (see (5.3.11) in \cite{MR0161012}) does not hold in the  model case
above neither for $\Sigma_{+}$ nor for $\Sigma_{-}$.
\par
An interesting phenomenon unraveled by Theorem \ref{thmik0} is that,
although there is no uniqueness across any of the characteristic hypersurfaces $\Sigma_{\pm }$, the fact that
the solution of the differential inequality is vanishing on the boundary of $\Omega$
and that  $\Sigma_{+}, \Sigma_{-}$ are intersecting transversally
is indeed producing  a uniqueness result.
\vs
\subsection{Statement of our result}
We consider a second-order differential operator $P$ in an open set $\mathcal U$
of $\R^{n}$ ($n\ge 3$)
with real principal symbol
\begin{equation}\label{operator}
p(x,\xi)=\poscal{Q(x)\xi}{\xi},
\end{equation}
where
$Q(x)$ is a real symmetric matrix, a $C^{1}$ function of $x$, with signature $(n-1,1)$,
i.e. with $(n-1)$ positive eigenvalues, 1 negative eigenvalue.
\par
Note that it is in particular the case of the Wave operator in a Lorentzian manifold
(the standard wave equation in $\R^{1+d}$ is
$c^{-2}\p_{t}^{2}-\Delta_{y}$ whose symbol is $-c^{-2}\tau^{2}+\val\eta^{2}$, indeed a quadratic form with signature 
$(n-1,1)$).
We assume also that the coefficients of the lower order terms in $P$ are bounded measurable.
\par\no
$\bullet$ 
Let $\phi_{+}, \phi_{-}:\mathcal U:\rightarrow \R$ be two $C^{2}$ functions such that 
 \begin{align}
 &d\phi_{+}\not=0 \text{ at }\phi_{+}=0,\qquad d\phi_{-}\not=0 \text{ at }\phi_{-}=0, \label{manifold}\\
&d\phi_{+}\wedge d\phi_{-}\not=0 \text{ at $\phi_{+}=\phi_{-}=0$}.\label{transverse}
\end{align}
$\bullet$ We define the open set 
\begin{equation}\label{omega}
\Omega=\{x\in \mathcal U,\ \phi_{+}(x)>0 \text{ and }\phi_{-}(x) >0\},
\end{equation}
and we have from the above assumptions
\begin{multline*}
\p\Omega=\{x\in \mathcal U,\ \phi_{+}(x)=0 \text{ and }\phi_{-}(x) \ge 0\}\\\cup\{x\in \mathcal U,\ \phi_{-}(x)=0 \text{ and }\phi_{+}(x) \ge 0\}.
\end{multline*}
$\bullet$ We shall also assume that the hypersurfaces $\Sigma_{+},\Sigma_{-}$ defined by 
\begin{equation}\label{hypersurfaces}
\Sigma_{\pm}=\{x\in \mathcal U, \phi_{\pm}(x)=0\},
\end{equation}
are both characteristic, i.e.
\begin{equation}\label{bothcar}
\forall x\in \Sigma_{+}, \poscal{Q(x)d\phi_{+}(x)}{d\phi_{+}(x)}=0,\quad
\forall x\in \Sigma_{-}, \poscal{Q(x)d\phi_{-}(x)}{d\phi_{-}(x)}=0.
\end{equation}
$\bullet$ Moreover, we shall assume that
\begin{equation}\label{sign}
\forall x\in \Sigma_{+}\cap \Sigma_{-},\quad \poscal{Q(x)d\phi_{+}(x)}{d\phi_{-}(x)}>0.
\end{equation}
Note that these assumptions hold in the model case of\  Section \ref{secik1}: with $$\phi_{\pm}=\val y-1\mp t, \quad \mathcal U=\{
(t,y)\in \R\times (\R^{d}\backslash\{0\})
\},$$
we have indeed for $y\not=0$ (which holds at $ \Sigma_{+}\cap \Sigma_{-}$ since there $t=0, \val{y}=1$), 
$$
d\phi_{+}\wedge d\phi_{-}=\bigl(\frac{y}{\val y}\cdot dy-dt\bigr)\wedge \bigl(\frac{y}{\val y}\cdot dy+dt\bigr)
=2\frac{y}{\val y}\cdot dy\wedge dt\not=0,
$$
and for this model-case we have 
$$
 \poscal{Q(x)d\phi_{+}(x)}{d\phi_{-}(x)}=-\frac{\p \phi_{+}}{\p t}   \frac{\p \phi_{-}}{\p t}+
 \frac{\p \phi_{+}}{\p y}\cdot \frac{\p \phi_{-}}{\p y}=1+\frac{y}{\val y}\cdot \frac{y}{\val y}=2.
$$
We are ready to state our unique continuation result.
\begin{theorem}\label{lerner} Let $\mathcal U, P, Q,  \Omega, \Sigma_{\pm},$ as above in 
%\eqref{operator}-\eqref{transverse}-\eqref{omega}--\eqref{hypersurfaces}-\eqref{bothcar}-\eqref{sign}.
\eqref{operator}---\eqref{sign}.
Let $u$ be a function in  $C^{2}(\overline \Omega)$ 
such that,
\begin{align}
 &\exists C\ge 0, \forall x\in \Omega,\quad\val{(Pu)(x)}\le C\bigl(\val{(\nabla_{x}u)(x)}+\val{u(x)}\bigr),\label{diff++}\\
 &u_{\vert \p\Omega}=0.\label{119}
\end{align}
 Then $u$ vanishes in a neighborhood of $\Sigma_{+}\cap \Sigma_{-}$.
\end{theorem}
\section{Proof of a differential inequality}
\subsection{Extending the solution}
Our function $u$ in Theorem \ref{lerner} belongs to $C^{2}(\overline\Omega)$ and thus there exists a function $\widetilde{u}\in C^{2}(\mathcal U)$ such that 
$$
\widetilde{u}_{\vert\overline\Omega}=u,\quad
\widetilde{u}=0\text{ at $\p\Omega$.}
$$
We may thus define a function $v$ as a function in $L^{\io}_{loc}(\mathcal U)$  by
\begin{equation}\label{functionv}
v(x)=\mathbf 1_{\overline\Omega}(x) \widetilde{u}(x), 
\end{equation}
and we have of course that 
\begin{equation}\label{}
\supp v\subset \overline \Omega,\quad v_{\vert\overline{\Omega}}=u.
\end{equation}
We claim now that $v$, which is defined ``globally'' in $\mathcal U$, does satisfy some differential inequality, at least in a neighborhood of $\Sigma_{+}\cap \Sigma_{-}$.
The sequel of this subsection is actually devoted
to the proof of such a differential inequality.
\vs
Taking $$(\underbrace{\phi_{+}}_{y_{1}}, \underbrace{\phi_{-}}_{y_{2}}, \underbrace{y_{3}}_{\in \R^{n-2}}),$$ as $C^{2}$ coordinates near a point $x_{0}$ of $\Sigma_{+}\cap \Sigma_{-}$,
we may consider there $\widetilde{u}$ as a $C^{2}$ function $\tilde u_{\kappa}(y_{1}, y_{2}, y_{3})$ defined by
$$
\tilde u_{\kappa}(y)= \widetilde{u}(\kappa(y)),
$$
where  $\kappa$ is a local $C^{2}$
diffeomorphism from $(-1,1)^{n}$ onto a neighborhood $\mathcal U_{0}$
of $x_{0}$, 
\begin{align*}
\tilde u_{\kappa}(0, y_{2},y_{3})=0\text{ \quad for $y_{2}\ge 0,$}\\
\tilde u_{\kappa}(y_{1}, 0,y_{3})=0\text{ \quad for $y_{1}\ge 0$}.
\end{align*}
We find that the principal symbol of the operator $P$ may be then written
as
$$
p_{\kappa}(y,\eta)=p\bigl(\kappa(y), \tr\kappa'(y)^{-1}\eta\bigr)=\poscal{Q(\kappa(y))\tr\kappa'(y)^{-1}\eta}{\tr \kappa'(y)^{-1}\eta},
$$ 
so that defining the symmetric $n\times n$ matrix 
$$
Q_{\kappa}(y)=\kappa'(y)^{-1} Q(\kappa(y))\tr\kappa'(y)^{-1}=\bigl(b_{jk}(y)\bigr)_{1\le j,k\le n},
$$
the fact that $\Sigma_{\pm}$ are characteristic may be expressed  by
$$
b_{11}(y)=b_{22}(y)=0,\quad \max_{1\le j\le n}\val{ y_{j}}<1.
$$
Note that since $Q$ is assumed to be $C^{1}$ and $\kappa$ is $C^{2}$ we still have the $C^{1}$ regularity for $Q_{\kappa}$.
As a result, the principal part $p_{\kappa}(y, D_{y})$ can be written (near 0)
as
\begin{multline}
-p_{\kappa}(y, D_{y})
= 2 b_{1,2}(y) \frac{\p^{2}}{\p y_{1}\p y_{2}} +\sum_{j\ge 3}2 b_{1,j}(y)
 \frac{\p^{2}}{\p y_{1}\p y_{j}} 
 +\sum_{j\ge 3} 2b_{2,j}(y)
 \frac{\p^{2}}{\p y_{2}\p y_{j}} 
 \\+\sum_{3\le j, k\le n}b_{j,k}(y)
 \frac{\p^{2}}{\p y_{j}\p y_{k}}. 
\end{multline}
Also we note that, with $\nu=\kappa^{-1}$ and $w\in C^{1}(\mathcal U_{0})$,  we have with $x=\kappa(y), y=\nu(x)$,
\begin{equation}\label{lot}
\frac{\p w}{\p x_{j}}=\sum_{1\le k\le n}\underbrace{\frac{\p y_{k}}{\p x_{j}}}_{\in C^{1}(\mathcal U_{0})}\frac{\p w_{\kappa} }{\p y_{k}},\qquad w_{\kappa}=w\circ\kappa.
\end{equation}
Let us now define with $H$ standing for indicator function  of $\R_{+}$ (Heaviside function),
$$
v_{\kappa}(y)=H(y_{1}) H(y_{2}) \tilde u_{\kappa}(y_{1}, y_{2}, y_{3}).
$$
\begin{lemma}\label{lem111}
Let $U$ be a $C^{2}$ real-valued function defined 
for $$(y_{1}, y_{2}, y_{3})\in (-1,1)\times (-1,1)\times (-1,1)^{n-2},$$
 and such that 
 \begin{align}
U(0, y_{2},y_{3})=0\text{ \quad for $y_{2}\ge 0,$}\label{0001}\\
U(y_{1}, 0,y_{3})=0\text{ \quad for $y_{1}\ge 0$}.\label{0002}
\end{align}
Then defining
$
V(y)=H(y_{1}) H(y_{2}) U(y_{1}, y_{2}, y_{3}),
$
we have 
\begin{align}
&\frac{\p V}{\p y_{1}}=H(y_{1})H(y_{2})\frac{\p U}{\p y_{1}},\quad 
\frac{\p V}{\p y_{2}}=H(y_{1})H(y_{2})\frac{\p U}{\p y_{2}},\label{111}\\
&\frac{\p^{2} V}{\p y_{1}\p y_{2}}=H(y_{1})H(y_{2})\frac{\p^{2} U}{\p y_{1}\p y_{2}},\label{222}\\
&\text{for $j\ge 3$, \ }\frac{\p^{2} V}{\p y_{1}\p y_{j}}=H(y_{1})H(y_{2})\frac{\p^{2} U}{\p y_{1}\p y_{j}}, \quad
\frac{\p^{2} V}{\p y_{2}\p y_{j}}=H(y_{1})H(y_{2})\frac{\p^{2} U}{\p y_{2}\p y_{j}},\label{333}\\
&\text{for $j, k\ge 3$, \ }\frac{\p^{2} V}{\p y_{j}\p y_{k}}=H(y_{1})H(y_{2})\frac{\p^{2} U}{\p y_{j}\p y_{k}}.\label{444}
\end{align}
\end{lemma}
\begin{proof}
We start with \eqref{111}: we note first that both sides of the equalities are making sense, left-hand sides as distribution derivatives of the $L^{\io}_{loc}$ function $V$
and 
 right-hand sides  as products of $C^{1}$ functions $\p U/\p y_{j}$ by the $L^{\io}_{loc}$ function $H(y_{1}) H(y_{2})$.
 We need only to use Leibniz formula using Assumptions \eqref{0001}, \eqref{0002},
 so that
 $$
 \frac{\p V}{\p y_{1}}= H(y_{1}) H(y_{2})\frac{\p U}{\p y_{1}}+\underbrace{\bigl(\delta_{0}(y_{1})\otimes H(y_{2})\bigr) }_{\text{measure}}\underbrace{U}_{\substack{\text{continuous}\\\text{function}}}.
 $$
 Let $\varphi\in\mooc\bigl((-1,1)^{n}=J\bigr)$.
 We have with brackets of duality $\langle \mathscr D'^{(0)}(J), C^{0}_{c}(J)\rangle $
 (here $\mathscr D'^{(0)}(J)$ stands for the Radon measures on $J$),
\begin{multline*}
\poscal{\bigl(\delta_{0}(y_{1})\otimes H(y_{2})\bigr) U}{\varphi}=
\poscal{\bigl(\delta_{0}(y_{1})\otimes H(y_{2})\bigr)}{U\varphi}
\\=\iint H(y_{2}) U(0,y_{2}, y_{3})\varphi(0,y_{2}, y_{3}) dy_{2}dy_{3}=0,
\end{multline*}
thanks to \eqref{0001};
similarly, we obtain the second formula in \eqref{111}, using \eqref{0002}.
Starting from the now proven \eqref{111}, Formula \eqref{333} is trivial and \eqref{444}
is an immediate consequence of the definition of $V$.
We are left with \eqref{222}: we have from \eqref{111}, noting that $\p U/\p y_{2}$ is a $C^{1}$ function,
\begin{multline}\label{penlem}
\frac{\p}{\p y_{1}}\left\{\frac{\p V}{\p y_{2}}\right\}=\frac{\p}{\p y_{1}}\left\{H(y_{1})H(y_{2})
\frac{\p U}{\p y_{2}}
\right\}
=\bigl(\delta_{0}(y_{1})\otimes H(y_{2})\bigr)\frac{\p U}{\p y_{2}}\\+
H(y_{1})H(y_{2})\frac{\p^{2}U}{\p y_{1}\p y_{2}},
\end{multline}
and with the above notations, for $\varphi\in \mooc(J)$,
we have 
\begin{multline}\label{lemlas}
\poscal{\bigl(\delta_{0}(y_{1})\otimes H(y_{2})\bigr) \frac{\p U}{\p y_{2}}}{\varphi}=
\poscal{\bigl(\delta_{0}(y_{1})\otimes H(y_{2})\bigr)}{ \frac{\p U}{\p y_{2}}\varphi}
\\=\iint H(y_{2}) \frac{\p U}{\p y_{2}}(0,y_{2}, y_{3})\varphi(0,y_{2}, y_{3}) dy_{2}dy_{3}.
\end{multline}
The identity \eqref{0001} holds for $(y_{2},y_{3})\in (0,1)\times (-1, 1)^{n-2}$, and thus the continuous function
$\p U/\p y_{2}$ satisfies 
$$
\frac{\p U}{\p y_{2}} (0, y_{2}, y_{3})=0\quad \text{for $(y_{2},y_{3})\in (0,1)\times (-1, 1)^{n-2}$},
$$
which implies from \eqref{lemlas}
$$
\bigl(\delta_{0}(y_{1})\otimes H(y_{2})\bigr) \frac{\p U}{\p y_{2}}=0,
$$
so that \eqref{penlem} yields the sought \eqref{222}.
 \end{proof}
 \begin{lemma}\label{lem222}
 Let $U$ be a real-valued $C^{2}$ function defined on $(-1,1)^{n}$ satisfying the assumptions of Lemma \ref{lem111}. Let $B(y)=(\beta_{jk}(y))_{1\le j\le k}$ be a real symmetric matrix of class $C^{1}$ on $(-1,1)^{n}$ 
 such that $\beta_{11}=\beta_{22}=0$.
 Then defining
\begin{equation}\label{213}
V(y)=H(y_{1}) H(y_{2}) U(y_{1}, y_{2}, y_{3}),
\end{equation}
we have  $\beta_{jk}\frac{\p^{2} V}{\p y_{j}\p y_{k}}\in L^{\io}_{loc}$ for all $j,k\in \{1,\dots,n\}$ and 
\begin{align}
&\bigl(\poscal{B(y) \p_{y}}{\p_{y}} V\bigr)(y)=H(y_{1})H(y_{2}) \bigl(\poscal{B(y) \p_{y}}{\p_{y}} U\bigr)(y),\label{214}\\
&(\frac{\p V}{\p y})(y)=H(y_{1})H(y_{2})(\frac{\p U}{\p y})(y).\label{215}
\end{align}
\end{lemma}
\begin{proof}[The proof of this lemma is an immediate consequence of Lemma \ref{lem111}]
 \end{proof}
\begin{lemma}\label{lem23}
 Let $U, V, B$ as in Lemma \ref{lem222} such that there exists a constant $C\ge 0$ such that 
 \begin{multline}\label{216}
 \forall y\in (-1,1)^{n}\text{ with $y_{1}>0$ and $y_{2}>0$}, \quad 
\val{\bigl(\poscal{B(y) \p_{y}}{\p_{y}} U\bigr)(y)}\\\le C\bigl(\val{(\nabla_{y} U)(y)}+\val{U(y)}\bigr).
\end{multline}
Then $\poscal{B(y) \p_{y}}{\p_{y}} V, \nabla_{y}V, V$ are locally bounded measurable and
 \begin{multline}\label{}
\text{for almost all } y\in (-1,1)^{n}, \quad 
\val{\bigl(\poscal{B(y) \p_{y}}{\p_{y}} V\bigr)(y)}\\\le C\bigl(\val{(\nabla_{y} V)(y)}+\val{V(y)}\bigr).
\end{multline}
\end{lemma}
\begin{proof}
 Using \eqref{213}, \eqref{214}, \eqref{215} we obtain that
 $\poscal{B(y) \p_{y}}{\p_{y}} V, \nabla_{y}V, V$ are locally bounded measurable and
\begin{align*}
\val{\bigl(\poscal{B(y) \p_{y}}{\p_{y}} V\bigr)(y)}&=H(y_{1})H(y_{2})
\val{\bigl(\poscal{B(y) \p_{y}}{\p_{y}} U\bigr)(y)}
\\
&\hs\le 
 H(y_{1})H(y_{2}) C\bigl(\val{(\nabla_{y} U)(y)}+\val{U(y)}\bigr)
 \\
 &\hs\hs=C\bigl(\val{(\nabla_{y} V)(y)}+\val{V(y)}\bigr).
\end{align*}
 \end{proof}
\begin{rem}\label{rem24}\label{warning}\rm
 The reader may note that we have used \eqref{0001},\eqref{0002} (i.e. vanishing of $U$ on the boundary of $\{y_{1}>0, y_{2}>0\}$ \underline{and} the fact that the boundary is characteristic, expressed in the $y$ coordinates by the equalities $\beta_{11}=\beta_{22}=0$. All these conditions are important to get Lemma \ref{lem222}.
 For instance, assuming only \eqref{0001},\eqref{0002}, but $\beta_{11}\not=0$, we would have to calculate
 \begin{align}
\frac{\p^{2}V}{\p y_{1}^{2}}&=\frac{\p}{\p y_{1}}\bigl\{\overbrace{(\delta_{0}(y_{1})\otimes H(y_{2})) U}^{=0\text{ from  \eqref{0001}}}+H(y_{1})H(y_{2})\frac{\p U}{\p y_{1}}\bigr\}\notag
\\
&=(\delta_{0}(y_{1})\otimes H(y_{2}))\frac{\p U}{\p y_{1}}+H(y_{1})H(y_{2})\frac{\p^{2} U}{\p y_{1}^{2}}\notag\\
&=(\delta_{0}(y_{1})\otimes H(y_{2}))\frac{\p U}{\p y_{1}}(0, y_{2},y_{3})+H(y_{1})H(y_{2})\frac{\p^{2} U}{\p y_{1}^{2}},
\label{noh2}
\end{align}
 and the term $\frac{\p U}{\p y_{1}}(0, y_{2},y_{3})$ can be non-zero (say for $U=y_{1}y_{2}$).
 The consequence of that situation would be that, even with $C=0$ in Assumption \eqref{216} and $\beta_{22}=0$,
 we would obtain the following equality for $V$
\begin{equation}\label{count}
\bigl(\poscal{B(y) \p_{y}}{\p_{y}} V\bigr)(y)=(\delta_{0}(y_{1})\otimes H(y_{2}))\frac{\p U}{\p y_{1}}(0, y_{2},y_{3})
\beta_{11}(0, y_{2},y_{3}),
\end{equation}
 where the rhs is a simple layer that cannot be controlled pointwise by $V$ or its first-order  derivatives.
 The fact that $V$ inherits a differential inequality from $U$, {\it assuming a simple vanishing of $U$} on the boundary $\p \Omega$, 
 is thus linked with the geometric situation: both hypersurfaces $\Sigma_{\pm}$ are characteristic for the operator $P$.
 Of course, we could have assumed a second-order vanishing (i.e. vanishing of the function and its first derivatives on $\p\Omega$, an assumption which would allow us to get rid of the rhs in \eqref{count} since then, we would have ${\p U}/{\p y_{1}}(0, y_{2},y_{3})=0$),
 but we would have lost most of the flavour of the model case given in Section \ref {secik1}.
\end{rem}
\begin{rem}\label{rem25}\rm
Let $u$ be a function satisfying the assumptions of Theorem \ref{lerner},
let $v$ be defined by \eqref{functionv} and let $x_{0}$ be a given point in $\Sigma_{+}\cap \Sigma_{-}$.
The assumptions \eqref{diff++}, \eqref{119} and Lemma \ref{lem23} imply that there exists a neighborhood  $\mathcal U_{0}$ 
of $x_{0}$
and a constant $C_{0}\ge 0$ such that
\begin{align}
 &\forall x\in \mathcal U_{0},\quad\val{(Pv)(x)}\le C_{0}\bigl(\val{(\nabla_{x}v)(x)}+\val{v(x)}\bigr),\label{diff++++}\\
 &\supp v\subset \overline \Omega,\quad v_{\vert\overline{\Omega}}=u.
\label{119+}
\end{align}
As a result, to prove Theorem \ref{lerner}, we are reduced
 to proving that $v$ vanishes on a neighborhood of $x_{0}$ in $\mathcal U_{0}$.
\end{rem}
\begin{rem}\rm
Building on Remark \ref{rem24}, we note that $u$ satisfies the diffe\-rential inequality
  \eqref{diff++}
 only on $\Omega$, which is not a convenient situation 
 to use H\"ormander's pseudo-convexity result.
 This is  the main reason for which we have introduced the new function $v$ given by \eqref{functionv}:
 we want to point out that the fact that $v$ is still satisfying some differential inequality (here \eqref{diff++++}) is a consequence of three geometrical facts:
 first of all,  $u$ is vanishing on $\p \Omega$, next, both hypersurfaces $\Sigma_{\pm}$ are characteristic for $P$ and finally $\Sigma_{\pm}$ are transverse.
 The function $v$ is defined in a neighborhood of $\Sigma_{+}\cap \Sigma_{-}$, supported in $\overline\Omega$ and a satisfies a differential inequality: we are in good position to use the classical Carleman estimates, provided we find a suitable pseudo-convex hypersurface.
\end{rem}
\subsection{The sign condition}
\begin{lemma}\label{lem26}
Let $\mathcal U,  Q,  \phi_{\pm}$ as given  in \eqref{operator}---\eqref{sign}.
We define 
\begin{align}\label{times}
\psi_{1}=\frac12\bigl(\phi_{+}+\phi_{-}\bigr),\qquad
\psi_{0}=\frac12\bigl(\phi_{-}-\phi_{+}\bigr).
\end{align}
Then we have ,
\begin{align}
&\poscal{Q(x)d\psi_{1}(x)}{d\psi_{1}(x)} >0>\poscal{Q(x)d\psi_{0}(x)}{d\psi_{0}(x)} \quad \text{at  }\Sigma_{+}\cap \Sigma_{-},\label{222+}\\
&\Omega=\{x\in \mathcal U, \psi_{1}(x)>\val{\psi_{0}(x)}\},\label{223}\\
&d\psi_{1}\wedge d\psi_{0}\not=0\quad \text{at  }\Sigma_{+}\cap \Sigma_{-}.\label{224}
\end{align}
\end{lemma}
\begin{proof}
Properties \eqref{223}-\eqref{224} are obvious.
Let us prove \eqref{222+}.
 From \eqref{bothcar},
 we have
 $$
 \poscal{Q(x)d\phi_{+}(x)}{d\phi_{+}(x)} =0\ \text{at $\Sigma_{+}$},\quad  \poscal{Q(x)d\phi_{-}(x)}{d\phi_{-}(x)} =0\ \text{at $\Sigma_{-}$},
 $$ 
 so that at $\Sigma_{+}\cap \Sigma_{-}$,
 since 
 $\phi_{\pm}=\psi_{1}\mp\psi_{0}$,
 $$
 \poscal{Q(x)d\psi_{1}(x)}{d\psi_{1}(x)} \mp 2\poscal{Q(x)d\psi_{1}(x)}{d\psi_{0}(x)} +\poscal{Q(x)d\psi_{0}(x)}{d\psi_{0}(x)} =0,
 $$
 implying
\begin{equation}\label{225}
 \poscal{Q(x)d\psi_{1}(x)}{d\psi_{1}(x)} +\poscal{Q(x)d\psi_{0}(x)}{d\psi_{0}(x)} =0=\poscal{Q(x)d\psi_{1}(x)}{d\psi_{0}(x)}.
\end{equation}
As a consequence, we have  from \eqref{sign},  at $\Sigma_{+}\cap \Sigma_{-}$, 
$$
0<\poscal{Q(x)d\phi_{+}(x)}{d\phi_{-}(x)} =
\poscal{Q(x)d\psi_{1}(x)}{d\psi_{1}(x)} -\poscal{Q(x)d\psi_{0}(x)}{d\psi_{0}(x)}, 
$$
and \eqref{225} yields
$$
0<2\poscal{Q(x)d\psi_{1}(x)}{d\psi_{1}(x)} 
\text{ and thus }\poscal{Q(x)d\psi_{0}(x)}{d\psi_{0}(x)} <0,
$$
which is the sought result.
\end{proof}
\begin{rem}\rm 
 Property \eqref{222+} means that, near $\Sigma_{+}\cap \Sigma_{-}$, 
 the hypersurface defined by $\{\psi_{0}=0\}$ is space-like
 whereas 
  the hypersurface defined by $\{\psi_{1}=0\}$ is time-like.
\end{rem}
\subsection{Finding a pseudo-convex hypersurface}
Let $p$ be given by \eqref{operator} and let $\psi$ be a $C^{2}$ function defined on $\mathcal U$ such that $d\psi\not=0$ at $\psi=0$.
We recall that the oriented hypersurface with equation $\{\psi=0\}$ is pseudo-convex with respect to $P$
whenever \eqref{hormander}
holds.
\begin{lemma}\label{lem28}
 Let $x_{0}\in \Sigma_{+}\cap \Sigma_{-}$ and let $\lambda>0$ be given.
 Then there exists a neighborhood $\mathcal U_{0,\lambda}$
 of $x_{0}$ in 
 $\mathcal U$ such that
 $$
 \Omega\cap \mathcal U_{0,\lambda}\subset\{x\in \mathcal U_{0,\lambda}, \psi_{1}(x)>\lambda \psi_{0}^{2}(x)\}.
 $$
\end{lemma}
\begin{proof}
 Assuming $\phi_{+}(x)>0, \phi_{-}(x)>0$, we find that
 $$
 \psi_{1}(x)>\val{\psi_{0}(x)}\text{ and thus } \psi_{1}(x)>\val{\psi_{0}(x)}\ge \lambda\psi_{0}^{2}(x),
 $$
 if $\val{\psi_{0}(x)}<1/\lambda$; since the latter condition defines a neighborhood of $x_{0}$, 
 we obtain the result.
\end{proof}
\begin{lemma}\label{lem29}
Let $n\ge 2$ and 
 let $M$ be an $n\times n$ real symmetric matrix with signature $(n-1,1)$.
 Let $\xi_{0},\xi$ be two non-zero vectors of $\R^{n}$
 such that
 $$
 \poscal{M\xi}{\xi}=0,\quad   \poscal{M\xi_{0}}{\xi_{0}}<0.
 $$
 Then 
 $\poscal{M\xi}{\xi_{0}}\not=0$.
\end{lemma}
\begin{proof}
There exists a $n\times n$
 invertible  matrix
$R$ such that
$$
\tr RMR= 
\begin{pmatrix}
 I_{n-1}&0\\
 0&-1
\end{pmatrix},
$$
where $I_{n-1}$ is the identity matrix with size $n-1$.
 Defining
 $$
 \eta_{0}=R^{-1}\xi_{0},\quad
  \eta=R^{-1}\xi,
 $$
 we get two non-zero vectors $\eta_{0}=(\zeta_{0},\tau_{0})$,
 $\eta=(\zeta,\tau)$ 
 of $\R^{n-1}\times \R$ such that,
with standard dot-products in $\R^{n-1}$,
\begin{equation}\label{888}
\val{\zeta}^{2}=\tau^{2},\quad 
\val{\zeta_{0}}^{2}<\tau_{0}^{2}.
\end{equation}
Note that since $(\zeta,\tau)\not=0$, the first equality implies that $\tau\not=0$.
 We have then $\poscal{M\xi}{\xi_{0}}=\zeta\cdot \zeta_{0}-\tau\tau_{0}$
 and if we had
 $\zeta\cdot \zeta_{0}=\tau\tau_{0}$,
 this would give
 $$
 \val{\tau\tau_{0}}=\val{\zeta\cdot \zeta_{0}}\le \val{\zeta}\val{\zeta_{0}}=\val{\tau}\val{\zeta_{0}}
 \underbrace{<}_{\substack{\tau\not=0\\\val{\zeta_{0}}<\val{\tau_{0}}}}\val{\tau}\val{\tau_{0}},
 $$
 which is impossible.
 As a result,
 we get indeed $\poscal{M\xi}{\xi_{0}}\not=0$.
\end{proof}
\begin{prop}\label{pro210}
  Let $x_{0}\in \Sigma_{+}\cap \Sigma_{-}$, and let $\psi_{0}, \psi_{1}$ be defined in Lemma \ref{lem26}. There exists $\lambda>0$ such that 
  $$
  \forall \xi\in \mathbb S^{n-1},\quad p(x_{0},\xi)=H_{p}{(\psi_{1}-\lambda \psi_{0}^{2})}(x_{0},\xi)=0\Longrightarrow
  H_{p}^{2}(\psi_{1}-\lambda \psi_{0}^{2})(x_{0},\xi)<0.
  $$
  This means  that the oriented hypersurface $S$
  with equation $\{\psi_{1}-\lambda \psi_{0}^{2}=0\}$
  is pseudo-convex with respect to the operator $P$ at $x_{0}$
  (note  that the first inequality in \eqref{222+} ensures that $S$ is non-characteristic for $P$).
\end{prop}
\begin{proof}
We have
$$
p(x,\xi)= \poscal{Q(x)\xi}{\xi},\quad \frac12H_{p}(\psi)(x,\xi)=\poscal{Q(x)\xi}{d\psi(x)},
$$
\begin{multline}\label{expression}
 \frac12H_{p}^{2}(\psi)(x,\xi)=\sum_{1\le j\le n}\frac{\p p}{\p \xi_{j}}\Bigl(\poscal{\frac{\p Q}{\p x_{j}}\xi}{d\psi(x)}
 +\poscal{Q(x)\xi}{d\frac{\p \psi}{\p x_{j}}}
 \Bigr)
 \\
 -\hskip-25pt\underbrace{\poscal{\frac{\p Q}{\p x}\xi}{\xi}}_{\substack{\text{covector with }\\\text{components
 $\poscal{\p_{x_{j}} Q(x) \xi}{\xi}$
 }}}\hskip-20pt\cdot \underbrace{ Q(x) d\psi(x)}_{\substack{\text{vector}}}.
\end{multline}
We note that 
$H_{p}^{2}(\psi)$
is a quadratic form in the variable $\xi$,
with coefficients depending polynomially on 
$Q(x), \nabla _{x}Q(x), \nabla _{x}\psi, \nabla_{x}^{2}\psi$.
\par
 Let $\xi\in \mathbb S^{n-1}$ such that at $x_{0}\in \Sigma_{+}\cap \Sigma_{-}$, 
 $
  p(x_{0},\xi)=H_{p}{(\psi_{1}-\lambda \psi_{0}^{2})}(x_{0},\xi)=0.
 $
 Since $\psi_{0}(x_{0})=0$, we infer that 
 \begin{equation}\label{999}
\poscal{Q(x_{0})\xi}{\xi}=\poscal{Q(x_{0})\xi}{d\psi_{1}(x_{0})}=0,\quad \text{i.e.}\quad
  p(x_{0},\xi)=H_{p}{(\psi_{1})}(x_{0},\xi)=0.
\end{equation}
 We have also 
 $$
 H_{p}^{2}(\psi_{1}-\lambda\psi_{0}^{2})=H_{p}^{2}(\psi_{1})-\lambda H_{p}^{2}(\psi_{0}^{2})
 $$
 and since $\psi_{0}(x_{0})=0$, we get 
\begin{equation}\label{key1}
  H_{p}^{2}(\psi_{1}-\lambda\psi_{0}^{2})(x_{0},\xi)=H_{p}^{2}(\psi_{1})(x_{0},\xi)-2\lambda \bigl(H_{p}(\psi_{0})(x_{0}, \xi)\bigr)^{2}.
\end{equation}
Moreover, we have, using \eqref{999} and \eqref{222+} 
$$ \poscal{Q(x_{0})\xi}{\xi}=\poscal{Q(x_{0})\xi}{d\psi_{1}(x_{0})}=0,\quad 
\poscal{Q(x_{0})d\psi_{0}(x_{0})}{d\psi_{0}(x_{0})}<0.
$$
We can use now Lemma \ref{lem29} (with $\xi, \xi_{0}=d\psi(x_{0}), M=Q(x_{0})$)
to obtain
$$\frac12
H_{p}(\psi_{0})(x_{0},\xi)=\poscal{Q(x_{0})\xi}{d\psi_{0}(x_{0})}\not=0.
$$
We may thus define the positive number
\begin{equation}\label{2123}
m_{0}=\min_{\substack{\xi\in \mathbb S^{n-1} \text{such that }\\
\poscal{Q(x_{0})\xi}{\xi}=\poscal{Q(x_{0})\xi}{d\psi_{1}(x_{0})}=0
}}
\overbrace{2\val {\poscal{Q(x_{0})\xi}{d\psi_{0}(x_{0})}}}^{=\val{H_{p}(\psi_{0})(x_{0},\xi)}}.
\end{equation}
Since we have
$$
p(x_{0}, \xi)=\poscal{Q(x_{0})\xi}{\xi},\quad H_{p}(\psi_{0})=2\poscal{Q(x_{0})\xi}{d\psi_{0}(x)},
$$
assuming with $m_{0}$ defined in \eqref{2123}, that 
\begin{equation}\label{key2}
\lambda >\frac{\max_{\xi\in \mathbb S^{n-1}}H_{p}^{2}(\psi_{1})(x_{0},\xi)}{2 m_{0}^{2}}=\lambda_{0},
\end{equation}
we obtain from  \eqref{key1}
that,
if for $\xi\in \mathbb S^{n-1}$, we have
$p(x_{0}, \xi)=H_{p}(\psi_{1}-\lambda\psi_{0}^{2})(x_{0},\xi)=0,$ then we get
$$
H_{p}^{2}(\psi_{1}-\lambda\psi_{0}^{2})(x_{0},\xi)<0,
$$
since \eqref{key2} implies for any $\xi\in \mathbb S^{n-1}$ such that $p(x_{0},\xi)=H_{p}(\psi_{1}-\lambda \psi_{0}^{2})(x_{0},\xi)=0$,
that $H_{p}(\psi_{0})(x_{0},\xi)\not=0$ and 
$$
\lambda>\frac{H_{p}^{2}(\psi_{1})(x_{0},\xi)}{2\bigl(H_{p}(\psi_{0})(x_{0},\xi)\bigr)^{2}}
\text{ and thus }
H_{p}^{2}(\psi_{1})(x_{0},\xi)-2\lambda\bigl(H_{p}(\psi_{0})(x_{0},\xi)\bigr)^{2}<0,
$$
yielding the sought result.
\end{proof}
\section{Proof of Theorem \ref{lerner} }
\subsection{A slightly different question}
Going back to the unique continuation question that we have to solve here, we may start looking again at our Remarks \ref{warning}, \ref{rem25}.
We have indeed  a  differential inequality on some open set $\mathcal U_{0}$
\begin{equation}\label{diff+++}
\val{(Pv)(x)}\le C_{0}\bigl(\val{(\nabla_{x}v)(x)}+\val{v(x)}\bigr),
\end{equation}
where $P$ is a second-order differential operator with $C^{1}$ coefficients and
also, 
defining
$$
\psi=\psi_{1}-\lambda \psi_{0}^{2},
$$
we know from Proposition \ref{pro210} that the hypersurface $\{\psi=0\}$ is pseudo-convex with respect to $P$.
Moreover
we know that the function $v$ is vanishing on the open set $\{x\in \mathcal U_{0}, \psi(x)<0\}$
(cf. Lemma \ref{lem28})
(maybe with a smaller neighborhood of $x_{0}$ than $\mathcal U_{0}$).
\par
It seems straightforward to apply now Theorem 8.9.1 in \cite{MR0161012} to obtain that $v$ should vanish near $\{\psi=0\}$ and give a positive answer to the question raised in Remark
\ref{rem25}.
However, we have to pay attention to the regularity at our disposal for the function $v$:
we know that $v, \nabla v$ are bounded measurable functions, so in particular $v$ belongs to the Sobolev space $H^{1}_{loc}(\mathcal U_{0})$. Nevertheless our Remark
\ref{warning} and the calculation \eqref{noh2} show that the simple vanishing of $u$ at $\p\Omega$ leaves open the possibility of having
for $\nabla^{2} v$ a simple layer so that we do not know if $v$ belongs to $H^{2}_{loc}(\mathcal U_{0})$,
but only
\begin{equation}\label{lessreg}
v\in H^{1}_{loc}(\mathcal U_{0}),\quad Pv \in L^{2}_{loc}(\mathcal U_{0}).
\end{equation}
Although these assumptions should be sufficient for the classical theorem to hold,
we have some checking to perform on this matter and we need to show that the classical assumption
$v\in H^{2}_{loc}(\mathcal U_{0})$
can be weakened down to \eqref{lessreg}.
\subsection{Invariance}
Let us start with checking some easy facts. In the first place, let 
$$
P=\poscal{Q(x) \p_{x}}{\p_{x}} +b(x)\cdot \p_{x}+c(x)
$$
be a differential operator with $C^{1}$ real coefficients in the principal part
($Q$ is a $C^{1}$ symmetric matrix) and $b, c$ are bounded measurable.
Then the differential inequality is invariant by a $C^{2}$
change of coordinates. Let 
$$
\kappa:\mathcal V_{0}\rightarrow \mathcal U_{0}\quad \text{be a $C^{2}$ diffeomorphism}.
$$
With standard notations we have
\begin{align}
&\frac{\p}{\p x_{j}}=\sum_{k}\frac{\p y_{k}}{\p x_{j}}\frac{\p }{\p y_{k}}, \quad \p_{x}=\tr \kappa'(y)^{-1}\p_{y},
\\
&\kappa^{*}(P)=\poscal{\kappa'(y) Q(\kappa(y)) \tr \kappa'(y)^{-1}\p_{y}}{\p_{y}} +b(\kappa(y))\cdot \tr\kappa'(y)^{-1}\p_{y}+c(\kappa(y)),
\end{align}
and since $\kappa'(y)$ is $C^{1}$, the new matrix 
$\kappa'(y) Q(\kappa(y)) \tr \kappa'(y)^{-1}$ is still $C^{1}$ and the lower order terms are bounded measurable,
whereas the differential inequality \eqref{diff+++}
becomes with $\kappa^{*}(v)=v\circ \kappa$, denoted by $v_{\kappa}$
$$
\val{\bigl(\kappa^{*} (P)v_{\kappa}\bigr)(y)}\le C_{0}\bigl(\val{(\tr\kappa'(y)^{-1}(\nabla_{y}v_{\kappa})(y)}+\val{v_{\kappa}(y)}\bigr),
$$
implying in the $y$-coordinates an inequality of the same type as  \eqref{diff+++}.
Assuming that the hypersurface with equation $\{\psi=0\}$ is non-characteristic is also invariant as well as the regularity of $\psi\circ \kappa$ if $\psi$ and $\kappa$ are assumed to be $C^{2}$.
\par
Also the pseudo-convexity Assumption \eqref{hormander}
is invariant by change of $C^{2}$ coordinates:
let us assume that it holds in the $x$-coordinates with 
$$
p(x,\xi)=\poscal{Q(x) \xi}{\xi},\quad \text{$Q$ real symmetric $C^{1}$ matrix, $\psi\in C^{2}$  such that \eqref{hormander} holds.}
$$
We have now
$$
p_{\kappa}(y,\eta)=\poscal{\kappa'(y)^{-1}Q(\kappa(y))\tr \kappa'(y)^{-1}\eta}{\eta}=p\bigl(\kappa(y), \tr \kappa'^{-1}(y)\eta\bigr),\quad \psi_{\kappa}=\psi\circ \kappa.
$$
Let us assume that 
$$
p_{\kappa}(y_{0}, \eta)=H_{p_{\kappa}}(\psi_{\kappa})(y_{0},\eta)=0, \eta \not=0.
$$
Assuming that $Q, \kappa$ are smooth functions, we find immediately
$$
H_{p_{\kappa}}^{2}(\psi_{\kappa})=\poi{p_{\kappa}}{\poi{p_{\kappa}}{\psi_{\kappa}}}=\bigl(H_{p}^{2}(\psi)\bigr)\circ \kappa.
$$
We may regularize $Q$  and $\kappa$ to get the same result for $Q$ of class $C^{1}$, $\kappa$ a $C^{2}$ diffeomorphism, using the expression \eqref{expression}
which shows  that 
$H_{p}^{2}(\psi)$
is a quadratic form in the variable $\xi$,
with coefficients depending polynomially on 
$Q(x), \nabla _{x}Q(x), \nabla _{x}\psi, \nabla_{x}^{2}\psi$.
\subsection{Mollifiers}
Thanks to Theorems 8.6.3 and 8.3.1 in \cite{MR0161012},
the pseudo-convexity hypothesis on $\psi$ expressed by Proposition \ref{pro210}
allows us to find a smooth real-valued function $\phi$ defined on a neighborhood $\mathcal U_{0}$
of $x_{0}$ and constants $C_{0}, \lambda_{1}$
such that
$$
d\phi\not=0,\quad \{x\in \mathcal U_{0}, \phi(x)<0\}\subset\{x\in \mathcal U_{0}, \psi(x)<0\},
$$
$$
\{x\in \mathcal U_{0},\phi(x)=0, \psi(x)\ge 0\}=\{x_{0}\},
$$
so that  all $w\in \mooc(\mathcal U_{0})$, for all
$\lambda\ge \lambda_{1}$,
\begin{equation}\label{carleman+}
C_{0}\norm{e^{-\lambda \phi} P w}_{L^{2}}\ge \lambda^{1/2}\norm{e^{-\lambda \phi} \nabla  w}_{L^{2}}
+ \lambda^{3/2} \norm{e^{-\lambda \phi} w}_{L^{2}}.
\end{equation}
Let $v$ be a $H^{1}_{loc}$ function satisfying \eqref{diff++++}, which implies that 
$Pv$ belongs to $L^{2}_{loc}$, let $\chi\in \mooc(\mathcal U_{0})$. Let $\rho\in \mooc(\R^{n})$
supported in the unit ball,  
with integral 1 and let us set for $\epsilon>0$,
\begin{equation}\label{moll}
\rho_{\epsilon}(z)=\epsilon^{-n}\rho(z/\epsilon).
\end{equation}
Carleman Inequality \eqref{carleman+} implies 
$$
C_{0}\norm{e^{-\lambda \phi} P (\rho_{\epsilon}\ast \chi v)}_{L^{2}}\ge \lambda^{1/2}\norm{e^{-\lambda \phi} \nabla   (\rho_{\epsilon}\ast \chi v)}_{L^{2}}
+ \lambda^{3/2} \norm{e^{-\lambda \phi}    (\rho_{\epsilon}\ast \chi v)}_{L^{2}}.
$$
Since $\phi$ is continuous and  $\chi v$ belongs to $H^{1}_{c}(\mathcal U_{0})$ as well as $\chi v\ast \rho_{\epsilon}$
 is supported in $\mathcal U_{0}$ for $\epsilon$ small enough,
we find by standard mollifying arguments that the rhs has the limit
$$
\lambda^{1/2}\norm{e^{-\lambda \phi} \nabla    \chi v}_{L^{2}}
+ \lambda^{3/2} \norm{e^{-\lambda \phi}  \chi v}_{L^{2}}.
$$
Checking the lhs,
we see that for $\chi_{0}\in \mooc(\mathcal U_{0})$, equal to 1 on a neighborhood of the support of $\chi$
(so that $\chi v\ast \rho_{\epsilon}$
 is supported in   
 $\{\chi_{0}=1\}$ for $\epsilon$ small enough)
 we find 
 $$
P (\rho_{\epsilon}\ast \chi v)=\chi_{0}P\hat \rho(\epsilon D) \chi v=[\chi_{0}P, \hat \rho(\epsilon D)] \chi v+\hat \rho(\epsilon D) \chi_{0}Pv.
$$
The term $\hat \rho(\epsilon D) \chi_{0}Pv$ goes to $\chi_{0}Pv$ in $L^{2}$ when $\epsilon$ goes to 0, since $Pv\in L^{2}_{loc}$.
We need to prove that the first term goes to 0 in $L^{2}$.
Proving this
will mean that  \eqref{carleman+} holds for $w=\chi v$ and unique continuation will follow
via standard arguments.
For that purpose, let us state and prove the following lemma.
\begin{lemma}
 Let $a\in C^{1}_{c}(\R^{n})$, let $v\in H^{1}_{c}(\R^{n})$. Then, with $\rho$ as above, we have
 $$
 \lim_{\epsilon\rightarrow 0_{+}} \Bigl(a \p_{x}^{2}(\rho_{\epsilon}\ast v)-\rho_{\epsilon}\ast\bigl( a\p_{x}^{2}v\bigr)\Bigr)=0,\quad\text{in $L^{2}(\R^{n})$}, 
 $$
 where $\p_{x}^{2} b$ stands for the Hessian matrix of $b$.
\end{lemma}
\noindent
{\it N.B.} The function $\p_{x}^{2}(\rho_{\epsilon}\ast v)$ is smooth compactly supported (thus in $L^{2}$)
and bounded in $H^{-1}$ when $\epsilon$ goes to 0.
As a result the product $a\p_{x}^{2}(\rho_{\epsilon}\ast v)$ is $C^{1}_{c}$ and  bounded in $H^{-1}$.
On the other hand $a\p_{x}^{2} v$ belongs to $H^{-1}_{c}$ and thus $\rho_{\epsilon}\ast (a\p_{x}^{2} v)$
is smooth compactly supported and bounded in $H^{-1}$.
\begin{proof}
 We have with obvious notations, for $v\in H^{1}_{c}$, 
\begin{align*}
D_{\epsilon}(v)&=\bigl( a \p_{x}^{2}(\rho_{\epsilon}\ast v)\bigr)(x)-\bigl(\rho_{\epsilon}\ast\bigl( a\p_{x}^{2}v\bigr)\bigr)(x)
\\&=
a(x)\int (\nabla \rho_{\epsilon})(x-y)(\nabla v)(y) dy 
-\overbrace{\int\underbrace{ \rho_{\epsilon}(x-y) a(y)}_{\in H^{1}(y)}\underbrace{\nabla^{2} v(y) }_{\in H^{-1}(y)}dy }^{\text{a bracket of duality}}
\\&=
\int a(x)(\nabla \rho_{\epsilon})(x-y)(\nabla v)(y) dy 
-\int (\nabla \rho_{\epsilon})(x-y) a(y)\nabla v(y) dy
\\&\hskip195pt+ \int \rho_{\epsilon}(x-y) \nabla a(y)\nabla v(y) dy.
\end{align*}
The very last term has limit  (the matrix)
$\nabla a\nabla v $ in $L^{2}(\R^{n})$.
Defining the first two terms as $\tilde D_{\epsilon}(v)$, we get 
\begin{align*}
\tilde D_{\epsilon}(v)&=\int\bigl(a(x)-a(y)\bigr)(\nabla \rho_{\epsilon})(x-y)\nabla v(y) dy\\
&=\int_{0}^{1}\int\nabla a(y+\theta(x-y))\cdot (x-y)(\nabla \rho_{\epsilon})(x-y)\nabla v(y) dyd\theta\\
&=\int_{0}^{1}\int\nabla a(y+\theta(x-y))\cdot \frac{(x-y)}{\epsilon}(\nabla \rho)_{\epsilon}(x-y)\nabla v(y) dy d\theta\\
&=\int_{0}^{1}\int\nabla a(x-\epsilon(1-\theta)z)\cdot z(\nabla \rho)(z)\nabla v(x-\epsilon z) dz d\theta.
\end{align*}
The function
$z_{k}{\p \rho}/{\p z_{j}}$ is smooth compactly supported with integral
$-\delta_{j,k}$
and thus 
\begin{align*}
\lim_{\epsilon\rightarrow 0_{+}}&
\int_{0}^{1}\int\p_{k}a(x-\epsilon(1-\theta)z)\cdot z_{k}(\p_{j} \rho)(z)\p_{l} v(x-\epsilon z) dz d\theta
\\
&=\lim_{\epsilon\rightarrow 0_{+}}\Bigl\{\int_{0}^{1}\int\bigl(\p_{k}a(x-\epsilon(1-\theta)z)-
\p_{k}a(x)
\bigr)\cdot z_{k}(\p_{j} \rho)(z)\p_{l} v(x-\epsilon z) dz d\theta
\\
&\hskip200pt+(\p_{k}a)(x)
\int z_{k}(\p_{j} \rho)(z)\p_{l} v(x-\epsilon z) dz \Bigr\}
\\
&=-(\p_{j}a)(x) (\p_{l} v)(x), \quad \text{in $L^{2}(\R^{n})$,}
\end{align*}
since
$$
\val{\p_{k}a(x-\epsilon(1-\theta)z)-
\p_{k}a(x)}\le \omega(\epsilon\val{z}),
$$
where $\omega$ is a modulus of continuity for $\nabla a$;
we have also used that for $\tilde\rho\in C^{0}_{c}(\R^{n})$ and $\tilde\rho_{\epsilon}$ defined by \eqref{moll},
we have for $w\in L^{2}(\R^{n})$ that $\tilde\rho_{\epsilon}\ast w$ belongs to $L^{2}(\R^{n})$
and 
$$
\lim_{\epsilon\rightarrow 0_{+}}(\tilde\rho_{\epsilon}\ast w)=J(\tilde \rho)w
\quad \text{in $L^{2}(\R^{n})$,\quad with $J(\tilde \rho)=\int_{\R^{n}}\tilde\rho(z) dz$.}
$$
As a result $\tilde D_{\epsilon}(v)$
 has limit $\nabla a\nabla v$ in $L^{2}(\R^{n})$, proving that 
 $D_{\epsilon}(v)$ has limit 0 in $L^{2}(\R^{n})$.
 \end{proof}
 %%%%%%%%%%%%%%%%


\begin{thebibliography}{10}

\bibitem{MR683803}
S.~Alinhac, \emph{Non-unicit\'e du probl\`eme de {C}auchy}, Ann. of Math. (2)
  \textbf{117} (1983), no.~1, 77--108. \MR{683803 (85g:35011)}

\bibitem{MR1363855}
S.~Alinhac and M.~S. Baouendi, \emph{A nonuniqueness result for operators of
  principal type}, Math. Z. \textbf{220} (1995), no.~4, 561--568. \MR{1363855
  (96j:35003)}

\bibitem{MR0104925}
A.-P. Calder{\'o}n, \emph{Uniqueness in the {C}auchy problem for partial
  differential equations.}, Amer. J. Math. \textbf{80} (1958), 16--36.
  \MR{0104925 (21 \#3675)}

\bibitem{MR0000334}
T.~Carleman, \emph{Sur un probl\`eme d'unicit\'e pour les syst\`emes
  d'\'equations aux d\'eriv\'ees partielles \`a deux variables
  ind\'ependantes}, Ark. Mat., Astr. Fys. \textbf{26} (1939), no.~17, 9.
  \MR{0000334 (1,55f)}

\bibitem{MR0117434}
Lars G{\.a}rding, \emph{Some trends and problems in linear partial differential
  equations}, Proc. {I}nternat. {C}ongress {M}ath. 1958, Cambridge Univ. Press,
  New York, 1960, pp.~87--102. \MR{0117434 (22 \#8213)}

\bibitem{MR0161012}
Lars H{\"o}rmander, \emph{Linear partial differential operators}, Die
  Grundlehren der mathematischen Wissenschaften, Bd. 116, Academic Press, Inc.,
  Publishers, New York; Springer-Verlag, Berlin-G\"ottingen-Heidelberg, 1963.
  \MR{0161012 (28 \#4221)}

\bibitem{MR1481433}
\bysame, \emph{The analysis of linear partial differential operators. {IV}},
  Grundlehren der Mathematischen Wissenschaften [Fundamental Principles of
  Mathematical Sciences], vol. 275, Springer-Verlag, Berlin, 1994, Fourier
  integral operators, Corrected reprint of the 1985 original. \MR{1481433
  (98f:35002)}

\bibitem{MR2033496}
\bysame, \emph{On the uniqueness of the {C}auchy problem under partial
  analyticity assumptions}, Geometrical optics and related topics ({C}ortona,
  1996), Progr. Nonlinear Differential Equations Appl., vol.~32, Birkh\"auser
  Boston, Boston, MA, 1997, pp.~179--219. \MR{2033496}

\bibitem{MR1996773}
\bysame, \emph{The analysis of linear partial differential operators. {I}},
  Classics in Mathematics, Springer-Verlag, Berlin, 2003, Distribution theory
  and Fourier analysis, Reprint of the second (1990) edition [Springer, Berlin;
  MR1065993 (91m:35001a)]. \MR{1996773}

\bibitem{MR2304165}
\bysame, \emph{The analysis of linear partial differential operators. {III}},
  Classics in Mathematics, Springer, Berlin, 2007, Pseudo-differential
  operators, Reprint of the 1994 edition. \MR{2304165 (2007k:35006)}

\bibitem{MR2470908}
Alexandru~D. Ionescu and Sergiu Klainerman, \emph{Uniqueness results for
  ill-posed characteristic problems in curved space-times}, Comm. Math. Phys.
  \textbf{285} (2009), no.~3, 873--900. \MR{2470908 (2010j:83075)}

\bibitem{MR794370}
David Jerison and Carlos~E. Kenig, \emph{Unique continuation and absence of
  positive eigenvalues for {S}chr\"odinger operators}, Ann. of Math. (2)
  \textbf{121} (1985), no.~3, 463--494, With an appendix by E. M. Stein.
  \MR{794370 (87a:35058)}

\bibitem{MR0097628}
Peter~D. Lax, \emph{Asymptotic solutions of oscillatory initial value
  problems}, Duke Math. J. \textbf{24} (1957), 627--646. \MR{0097628 (20
  \#4096)}

\bibitem{MR0170112}
Sigeru Mizohata, \emph{Some remarks on the {C}auchy problem}, J. Math. Kyoto
  Univ. \textbf{1} (1961/1962), 109--127. \MR{0170112 (30 \#353)}

\bibitem{MR1614547}
Luc Robbiano and Claude Zuily, \emph{Uniqueness in the {C}auchy problem for
  operators with partially holomorphic coefficients}, Invent. Math.
  \textbf{131} (1998), no.~3, 493--539. \MR{1614547 (99e:35004)}

\bibitem{MR850550}
Xavier Saint~Raymond, \emph{L'unicit\'e pour les probl\`emes de {C}auchy
  lin\'eaires du premier ordre}, Enseign. Math. (2) \textbf{32} (1986),
  no.~1-2, 1--55. \MR{850550 (87m:35041)}

\bibitem{MR1081811}
Christopher~D. Sogge, \emph{Strong uniqueness theorems for second order
  elliptic differential equations}, Amer. J. Math. \textbf{112} (1990), no.~6,
  943--984. \MR{1081811 (91k:35068)}

\bibitem{MR0330739}
Monty Strauss and Fran{\c{c}}ois Tr{\`e}ves, \emph{First-order linear {PDE}s
  and uniqueness in the {C}auchy problem}, J. Differential Equations
  \textbf{15} (1974), 195--209. \MR{0330739 (48 \#9076)}

\bibitem{MR1326909}
Daniel Tataru, \emph{Unique continuation for solutions to {PDE}'s; between
  {H}\"ormander's theorem and {H}olmgren's theorem}, Comm. Partial Differential
  Equations \textbf{20} (1995), no.~5-6, 855--884. \MR{1326909 (96e:35019)}

\end{thebibliography}
\end{document}